\documentclass[12pt]{article} 
\usepackage[letterpaper,top=1.5cm, bottom=1.6cm, left=1.6cm, right=1.6cm]{geometry}

\usepackage[latin9]{inputenc}
\usepackage{bbm}
\usepackage[normalem]{ulem}
\usepackage{amssymb}
\usepackage{amsmath}
\usepackage{amsfonts}
\usepackage{amsthm}
\usepackage{latexsym}
\usepackage[usenames]{color} 
\usepackage{bm}
\usepackage{overpic}
\usepackage{graphicx}
\usepackage{mathrsfs}
 \usepackage{enumerate}

\usepackage{exscale}

\usepackage{mathtools}
\usepackage{hyperref}
\def\CBV{\text{\sl CBV}}

\def\ignore#1{}

\def\wt{\widetilde}

\newcommand{\ud}{\,\mathrm{d}}
\def\eps{\varepsilon}
\def\Ind#1{{\mathbbmss 1}_{_{\scriptstyle #1}}}

\definecolor{DarkGreen}{rgb}{0.2,0.6,0.2}

\def\blue#1{\textcolor{black}{#1}}

\def\green#1{\textcolor{DarkGreen}{#1}}
 \def\green#1{}

\title{The associativity rule in pathwise functional It\^o calculus}
\author{ \normalsize Alexander Schied\thanks{
		Department of Statistics and Actuarial Science, University of Waterloo. E-mail: {\tt alex.schied@gmail.com}} \and\setcounter{footnote}{6}\normalsize Iryna Voloshchenko\footnote{Department of Mathematics, University of Mannheim. E-mail: {\tt irynaice@gmail.com}}}

\date{ \normalsize May 21, 2018}

\begin{document}	

\maketitle 

\begin{abstract} 
In this paper we establish the associativity property of the pathwise It\^o integral in a functional setting for continuous   integrators. 
Here, associativity refers to the computation of the It\^o differential of an It\^o integral, by means of the intuitive cancellation of the It\^o differential and integral signs.   \end{abstract}

\theoremstyle{definition} 
\newtheorem{definition}{Definition}[section]
\newtheorem{rem}{Remark}[section]
\newtheorem{ex}{Example}[section]
\newtheorem{pr}{Problem}[section]
\newtheorem{assump}{Assumption}[section]

\theoremstyle{plain}
\newtheorem{thm}{Theorem}[section]
\newtheorem{lemma}{Lemma}[section]
\newtheorem{prop}{Proposition}[section]
\newtheorem{cor}{Corollary}[section]

\section{Introduction} 

 Dupire~\cite{Dupire} and  Cont and Fourni\' e~\cite{CF,CF2} have recently introduced a new type of stochastic calculus,  known as \emph{functional It\^o calculus}. It is based on an extension  of the classical  It\^o formula to functionals depending on the entire past evolution of the underlying path, and not only on its current value.  
The approach taken in~\cite{CF} is a direct extension of the non-probabilistic It\^o formula of  F\"ollmer~\cite{Ito_F} to non-anticipative functionals on Skorokhod space. These functionals are required to possess certain directional derivatives which may be computed pathwise, but no Fr\' echet differentiability  is imposed. 
An alternative approach, which to some extent still relies on probabilistic arguments, was introduced by  Cosso and  Russo~\cite{Russo}; it  is based on the theory of stochastic calculus via regularization~\cite{Russo2,Fabbri,GR1,GR2,GR3,GR4}.

In recent years, pathwise It\^o calculus has been particularly popular in mathematical finance and economics; see, e.g.,~\cite{bender,bick,Davis,F2001,F&S,Schied07,SPVol15,Sondermann}.
This is due to the fact that  the results derived with the help of pathwise It\^o calculus  are robust with respect to model risk that might stem from a misspecification of  probabilistic dynamics.  The only assumption on the underlying paths is that they admit  the quadratic variation in the sense of~\cite{Ito_F}.

Our first contribution in this paper is a slight extension of the functional change of variables formula from~\cite{CF},  which 
is motivated by the fact that  functionals of interest often depend on additional arguments such as quadratic variation, moving average, or running maximum of the underlying path. These quantities, however, are often not regular enough to fit fully into the framework of~\cite{CF} (see also the discussions in~\cite{Ananova,CF2,Dupire}).   
To this end, we will allow our functionals $F$ to depend on an additional variable $A$ that corresponds to a general path of bounded variation,  and then  extend the notions of  the horizontal and vertical derivatives to functionals of this type.
This extension  will also be crucial for the proof of our associativity rule. 

Our main result will be a functional version of the associativity rule for the pathwise It\^o integral. To describe this result informally, consider a continuous path $X:[0,T]\to\mathbb R$ that admits the continuous quadratic variation  in the sense of F\"ollmer \cite{Ito_F}. Then, if $\xi$ is an integrand for which we can form the pathwise It\^o integral  $Y(t):= \int_0^t \xi(s)\ud X(s)$, and $\eta$ is another integrand for which we can form $\int_0^t\eta(s)\ud Y(s)$,  we have the intuitive cancellation property
$$
\int_0^T \eta (s)\ud  Y(s) = \int_0^T \eta(s)\xi(s)\ud  X(s). 
$$
This cancellation rule is often called the \emph{associativity} of the integral.
 Note that in standard stochastic calculus it follows immediately from an application of the Kunita--Watanabe characterization of the stochastic integral.  In the pathwise setting, however, an analogue of the Kunita--Watanabe characterization is not available, and the proof of the associativity property becomes surprisingly involved, as only analytical tools are at our disposition.   In \cite[Theorem 13]{Schied13} a pathwise associativity result was obtained in the case in which $\xi(t)$ and $\eta(t)$ are functions of $X(t)$ and $Y(t)$, respectively. In our present functional context, $\xi(t)$ and $\eta(t)$ may now depend on the stopped \emph{paths} $(X(s\wedge t))_{s\in[0,T]}$ and $(Y(s\wedge t))_{s\in[0,T]}$. The crucial difference to the situation considered in \cite{Schied13} is that this functional dependence must be retained by writing $\xi(t,X)$ and $\eta(t,Y)$. This dependence must moreover  satisfy several regularity properties, because  the corresponding It\^o integrals are based on Riemann sums of  $\xi(t,X^n)$ and $\eta(t,Y^n)$, where $X^n$ and $Y^n$ are approximations of $X$ and $Y$. Our corresponding result is  Theorem~\ref{assoccont}. It is the main result of this paper. Its proof uses an entirely different strategy than the one for \cite[Theorem 13]{Schied13}, and, when put into the context of \cite{Schied13}, can also considerably 
 simplify that proof. Our result moreover corrects some errors in \cite{Voloshchenko}.

Just as in standard stochastic calculus, associativity is a fundamental property of the It\^o integral and crucial for many applications. For instance, in~\cite{Schied13}, a basic version of  the associativity rule was derived so as to give a pathwise treatment of constant-proportion portfolio insurance strategies (CPPI). Our original motivation for deriving an associativity rule within functional It\^o calculus  stems from the fact that it is helpful for analyzing functionally dependent strategies in a pathwise version of stochastic portfolio theory; see our companion paper~\cite{SPVol15}.

The paper is organized as follows. 
In Section~\ref{Sect2.1} we provide the basic notions for non-anticipative functionals in a way which slightly extends the notions introduced in \cite{Dupire} and \cite{CF}.  In Section~\ref{subsect 2.2} we provide a corresponding change of variables formula  for non-anticipative functionals depending  on an additional bounded variation component. With this at hand, we can state and show in Section~\ref{associativity section} the associativity rule for the pathwise functional It\^o integral. All proofs are given in Section~\ref{proofs section}.

\section{Preliminaries and statement of results} \label{Model} 

\subsection{Non-anticipative functionals and functional derivatives on spaces of paths} \label{Sect2.1} 

In the following, we will  first describe our framework. We essentially follow~\cite{Dupire,CF} and  slightly modify and extend the definitions and notations given there. This applies in particular to the Definitions~\ref{regularity properties def},~\ref{horizontal der def},~\ref{Def.grad}, and~\ref{Cjk}.

In the sequel, we fix $T>0$ and  open sets $U\subset  \mathbb{R}^d$  and $S \subset \mathbb{R}^m$. By $\text{\sl CBV}([0,T],S)$ we will denote the set of those continuous $S$-valued functions whose components are of bounded variation on $[0,T]$. We will write $D([0,T],U)$ for the usual Skorokhod space of $U$-valued c\`{a}dl\`{a}g functions $X$ on $[0,T]$ with left limits $X(t-):=\lim_{s\uparrow t}X(s)\in U$. For $X\in D([0,T],U)$ and $t\in [0,T]$, we let $X^t=(X(t\wedge s))_{s\in[0,T]}$  denote the  path stopped in $t$. The space $D([0,T],U)$ will be equipped with the following supremum norm,
\begin{equation} 
\|X\|_\infty=\sup_{u\in[0,T]}\arrowvert X(u)\arrowvert \qquad\text{for $X\in   D([0,T],U) $.}
\label{dist}\end{equation}
 A functional $F:[0,T]\times D([0,T],U)\times   \text{\sl CBV}([0,T],S)  \mapsto \mathbb{R}$ is called \emph{non-anticipative} if 
\begin{equation}
 F(t, X, A)=F(t, X^t, A^t)
\label{nonanticip}\end{equation}
for all $(t,X, A)\in [0,T]\times D([0,T],U)\times   \text{\sl CBV}([0,T],S)$.
This definition of non-anticipativity is analogous to the one in \cite{CF}. 
In addition, however, it allows $F$ to depend on the path $A$ of bounded variation. This additional path should not be confused with the dependence on the arbitrary c\`{a}dl\`{a}g path $v$ in \cite{CF}, which due to its predictable dependence does not give rise to additional derivatives. Examples for $A$ we have in mind are a running maximum $A(t)=\max_{u\leq t} X(u)$  or the quadratic variation $A(t)=[X](t)$ of a trajectory $X$, which may not be  absolutely continuous in $t$. Functionals depending on such  quantities are not directly covered by the It\^o formula from~\cite{CF}.  The ability to deal with such functionals, however, will be crucial for  our proof of the associativity rule. They also naturally appear in many applications to mathematical finance. The following definition recalls regularity properties introduced in \cite{CF} 
and presents them in our slightly modified setup.

\begin{definition} \label{regularity properties def}Let $F$ be a non-anticipative functional on $[0,T]\times D([0,T],U)\times\text{\sl CBV}([0,T],S)$.
\begin{enumerate}
\item $F$ 
is called \emph{boundedness-preserving}   if for every $A\in \text{\sl CBV}([0,T],S)$ and any compact subset $K\subset U$  there exist a constant $C$ such that 
\begin{equation} 
\arrowvert F(t,X,A)\arrowvert \le  C
 \label{bdd}\end{equation}
for all $t\in [0,T]$ and  $X\in D([0,T],K)$.

\item $F$ is called \emph{continuous at fixed times}, if for all $\eps>0$, $t\in[0,T]$, $X\in  D([0,T],U)$, and $A\in \text{\sl CBV}([0,T],S)$, there exists $\eta>0$ such that $|F(t,X,A)-F(t,Y,A)|<\eps$ for all $Y\in D([0,T],U)$ for which $\|X^t-Y^t\|_\infty<\eta$.

\item   $F$ 
is called \emph{left-continuous}  if for all $t\in(0,T]$, $\eps>0$, $X\in D([0,T],U)$, and $A\in \text{\sl CBV}([0,T],S)$, there exists $\eta>0$ such that 
 \begin{equation} 
|F(t,X,A)-F(t-h,Y,A)|<\eps
\label{leftcont}\end{equation}
for all $h\in[0,\eta)$ and $Y\in D([0,T],U)$ for which
$\|X^t-Y^{t-h}\|_\infty<\eta$.

\end{enumerate}
  \end{definition}
  
  \begin{rem}\label{boundedness preserving rem}Suppose that the non-anticipative functional $F$ is  boundedness-preserving. Then for all $X\in D([0,T],U)$ and $A\in \text{\sl CBV}([0,T],S)$ there are $C>0$ and $\eta>0$ such that $\|X-Y\|_\infty<\eta$ implies that 
  \begin{equation} 
\sup_{t\in[0,T]} \arrowvert F(t, Y,A)\arrowvert\leq C.
  \label{locbdd}\end{equation}

  \end{rem}

 The additional dependence of $F(t,X,A)$ on the component $A\in\CBV([0,T],S)$ gives rise to the following notion of a horizontal derivative, which extends the corresponding notion from  \cite{Dupire} and   \cite{CF}.

\begin{definition}[\textbf{Horizontal derivative}] \label{horizontal der def}
Let $F$ 
be a non-anticipative functional on  $[0,T]\times  D([0,T],U)\times \text{\sl CBV}([0,T],S) $. We say that $F$ is \emph{horizontally  differentiable}, 
if there exist  non-anticipative and boundedness preserving functionals $\mathscr{D} _0F, \mathscr{D} _{1}F,\dots, \mathscr{D} _{m}F$ on $[0,T]\times  D([0,T],U)\times \text{\sl CBV}([0,T],S) $ such that for $0\le s<t\le T$ and 
$(X,A)\in  D([0,T],U)\times \text{\sl CBV}([0,T],S) $, the functions $[s,t]\ni r\mapsto \mathscr{D} _iF(r,X^s,A)$ are Borel measurable and
\begin{align*}
 F(t,X^s,A)-F(s,X^s,A)=\sum_{i=0}^m\int_s^t\mathscr{D} _iF(r,X^s,A)\, A_i(dr),
\end{align*}
where we put $A_0(r):=r$.
In this case, the vector-valued functional
$$\mathscr{D}F=\left(\mathscr{D} _0F, \mathscr{D} _{1}F,\dots, \mathscr{D} _{m}F\right) $$
is called the \emph{horizontal derivative} of $F$.
\end{definition}

\begin{rem} In Definition~\ref{horizontal der def}, the horizontal derivative was defined as a Radon--Nikodym derivative with respect to the signed vector measure $(\ud r,\ud A_1(r),\dots,\ud A_m(r)) $. It follows, e.g., from \cite{Jeffrey} that $F$ will be horizontally differentiable with horizontal derivative $\mathscr{D}F=\left(\mathscr{D} _0F, \mathscr{D} _{1}F,\dots, \mathscr{D} _{m}F\right) $ if the following limits exist for all $(t,X,A)$ and if they give rise to locally bounded and non-anticipative functionals  on $[0,T]\times  D([0,T],U)\times \text{\sl CBV}([0,T],S) $ satisfying the measurability requirement from Definition~\ref{horizontal der def},
\begin{align}
\mathscr{D} _0F(t,X^t,A^t)&=\lim_{h\downarrow 0}\frac{F(t+h,X^{t},A^{t})-F(t,X^{t},A^{t})}{h}\label{Russomodif}\\
 \mathscr{D} _{k}F(t, X^{t}, A^{t})&=\lim_{h\downarrow 0}\frac{F(t,X^{t},A_1^{t},\dots, A_k^{t+h},\dots,A_m^{t})-F(t,X^{t},A^{t})}{A_k(t+h)-A_k(t)} \Ind{\{A_k(t+h)\neq A_k(t)\}},
 \label{D-1} 
 \end{align}
for $ k=1,\;\dots, m$. If $F$ does not depend on $A$, then \eqref{Russomodif} yields the horizontal derivative defined in \cite{CF}.
\end{rem}

The following definition is identical to the one given in \cite{Dupire,CF}.

\begin{definition}[\textbf{Vertical derivative}] \label{Def.grad}
 
A non-anticipative functional $F$ 
is said to be \emph{vertically differentiable} at $(t,X,A) $ if the map
$
\mathbb{R}^d\ni v\to F(t, X+v\Ind{[t,T]}, A^t)
$
is differentiable at $0$. The \emph{vertical derivative} of $F$ at $(t,X,A) $ will then be the gradient of that map at $v=0$. It will be denoted by
$$  \nabla_X F(t,X,A)=\left( \partial _i F(t,X,A)\right)_{i=1,\dots, d},
$$
where 
$$ \partial _i F(t,X,A)=\lim_{h\to0}\frac{F(t,X+he_i\Ind{[t,T]},A)-F(t,X,A)}{h}.
$$
If $F$ is vertically differentiable at all $(t,X,A)$, the map
\begin{align}
\nabla_X F:[0,T]\times D([0,T],U)\times  \text{\sl CBV}([0,T],S) &\longmapsto \mathbb{R}^d\notag\\
 (t,X,A)&\longrightarrow  \nabla_X F(t,X,A)\label{grad1} 
 \end{align}
is a non-anticipative functional  with values in  $ \mathbb{R}^d$ and called the \emph{vertical derivative of $F$}.
\end{definition}

%
\begin{ex}\label{Regularity example}
Sometimes, a quantity of interest can either be considered as a path-dependent functional of  $X\in D([0,T],U)$ only or as an additional trajectory of bounded variation. The latter possibility allows us to include functionals that may not be regular enough for the setting of~\cite{CF} or~\cite{Dupire}.
This illustrates one advantage of our extended approach. See also \cite[Section 3]{SPVol15} for a discussion of a related but more sophisticated situation in the context of model-free stochastic portfolio theory. For the following examples let us assume $d=1$. \begin{enumerate}
\item Consider the time average  of $X$, 
$$F(t, X)=\int_0^t X(s)\ud s.$$
Alternatively, this functional can be represented as  
$$G(t,X,A)=G(t,A)=A(t)\qquad\text{for $A(t):=\int_0^t X(s)\ud s$.}$$  
In the first approach, we have $\mathscr{D}F(t,X)=X(t-)$ and $\nabla_X F(t,X)=0.$ 
In the second approach, we have $\mathscr{D}G(t,X,A)=(0,1)  $  and $\nabla _XG(t,A,X)=0.$ Thus, both approaches work here.

\item Consider the running maximum of the first component,
$$F(t, X)=\underset{0 \leq s \leq t}{\max} \, X(s).$$
Alternatively, this functional can be represented as  
$$G(t,X,A)=G(t,A)=A(t)\qquad\text{for $A(t):=\max\limits_{0 \leq s \leq t}  X(s).$}$$
Then, $F$ is not (vertically) differentiable in the first approach, and we would have to resort to smoothing techniques~\cite{Dupire}.
In the second approach, however, the horizontal derivative  \blue{exists} and we have as before that $\mathscr{D}G(t,A)=(0,1) $  and  $\nabla_XG(t,A)=0.$ Functionals involving the running maximum appear in  mathematical finance, e.g., for lookback or barrier options and functionals involving the maximal drawdown.

\item Consider the functional 
\blue{$$F(t, X)=[X^c](t)+\sum_{s\le t}(\Delta_sX)^2,$$
where $\Delta_sX:=X(s)-X(s-)$ and $[X^c]$ is the pathwise quadratic variation of the continuous part $X^c$ of $X$ (see Definition~\ref{QV1d}). Note that $F$ is defined only on a suitable subset of $D([0,T],\mathbb R)$.}
Alternatively, this functional can be represented as 
$$G(t,X,A)=G(t,A)=A(t)\qquad\text{for $A(t):=[X](t).$}$$
In the first approach, we have  $\mathscr{D}F(t,X)=0$ and $\nabla_X F(t,X)=2(X(t)-X(t-))$.  In particular, we have  $\mathscr{D}F(t,X)=0$ and $\nabla_X F(t,X)=0$ if $X$ is continuous, so that this approach may not work.  In the second approach, we have again  $\mathscr{D}G(t,A)=(0,1) $ and $\nabla_X G(t,A)=0.$ Functionals involving the quadratic variation appear in mathematical finance, e.g., for options on realized variance or volatility \cite{Davis}.
\end{enumerate}
\end{ex}
If the functional $F$ admits  horizontal and vertical derivatives $ \mathscr{D} F$ and $ \nabla_X F$, we may iterate the corresponding operations so as to define \blue{higher-order} horizontal and vertical derivatives. Note that our horizontal derivatives are boundedness-preserving by definition. The following definition is a modification and extension of \cite[Definition 9]{CF}.

\begin{definition} \label{Cjk}
We denote by $\mathbb{C}^{1,2}_b(U,S)$ the set of all non-anticipative functionals $F
$ on $[0,T]\times D([0,T],U)\times\text{\sl CBV}([0,T],S)$ such that:
\begin{enumerate}
\item $F$ is left-continuous, horizontally differentiable, and twice vertically differentiable;
\item the horizontal derivative $\mathscr DF$ is  continuous at fixed times;
\item the vertical derivatives $\nabla_XF$ and $\nabla^2_XF$ are left-continuous and boundedness-preserving.
\end{enumerate}
\end{definition}

\begin{rem}Taking $h=0$ in Definition~\ref{regularity properties def} (c), we see that part (c) in Definition~\ref{Cjk} implies the continuity of the function $\mathbb R^d\ni v\mapsto F(t,X+v\Ind{[t,T]},A)$. Thus, Schwarz's theorem implies that the second partial vertical derivatives,
$$\partial_{ij}F(t,X,A):=\partial_i\big(\partial_j F(t,X,A)\big)
$$
are symmetric: 
$$\partial_{ij}F(t,X,A)=\partial_{ji}F(t,X,A),\qquad i,j=1,\dots, d.
$$
\end{rem}

\subsection{Functional It\^o formula with additional components of bounded variation} \label{subsect 2.2}

In this section, we present a functional It\^o formula for functionals $F\in\mathbb{C}^{1,2}_b(U,S)$, which slightly extends the functional It\^o formulas from \cite{Dupire} and \cite{CF}. This extension will be needed in particular for the proof of our associativity formula. It also has several other potential applications, notably in mathematical finance. Before stating this result, we recall now the notion of quadratic variation in the sense of F\"ollmer~\cite{Ito_F}. To this end, we fix from now on  a refining sequence of partitions  $\left(\mathbb{T}_n\right)=\left(\mathbb{T}_n\right)_{n\in\mathbb{N}}$ of $[0,T]$. That is, each $\mathbb T_n$ is a finite subset of $[0,T]$ such that $0,1\in \mathbb T_n$, we have $\mathbb T_1\subset\mathbb T_2\subset\cdots$, and the mesh of $\mathbb T_n$ tends to zero as $n\uparrow\infty$. For fixed $n$ and given $t\in\mathbb T_n$, we denote by  $t'$ the successor of $t$ in $\mathbb{T}_n$, i.e.,  
$$t'=\begin{cases}\min\{u\in\mathbb{T}_n\,|\,u>t\}&\text{if $t<T$,}\\
T&\text{if $t=T$.}
\end{cases}
$$

\begin{definition}[\bf{Quadratic variation}]\label{QV1d}
Let $X\in C([0,T], \mathbb{R}^d)$.
\begin{enumerate}
\item If $d=1$,
  we say that $X $ admits the  \emph{continuous quadratic variation} $[ X]$ along $\left(\mathbb{T}_n\right)_{n\in\mathbb{N}}$ if for all $t\in(0,T]$
  the sequence 
  \begin{equation} 
 \sum_{s\in \mathbb{T}_n\atop s\le t} (X(s')-X(s))^2\label{covariation def}\end{equation}
converges to a finite limit, denoted $[ X] (t)$, and if $t\mapsto [ X] (t)$ is continuous for $[X](0):=0$. 
\item Let $d>1$. 
We say that  $X$ admits the \emph{continuous quadratic variation} $[X]$ along  $\left(\mathbb{T}_n\right)$ if all real-valued functions $X_i+X_j,\; i,j=1,\dots,d,$ admit the continuous quadratic variation $[X_i+X_j]$ in the sense of (a). In this case, we set $[X]:=([X_i,X_j])_{i,j=1,\dots,d}$ for 
\begin{align} 
[ X_i,X_j](t):=\frac{1}{2}\Big([ X_i+X_j](t)  - [ X_i](t)-[ X_j] (t)\Big).
\label{QVm}\end{align}
\end{enumerate}

\end{definition}

\bigskip

In the context of (b), we clearly have $[X_i,X_i]=[X_i]$.
Note moreover that the quadratic variation depends on the  choice of the refining sequence of partitions $(\mathbb T_n)$ and that the existence of the quadratic variations $[X_i]$ and $[X_j]$ does not imply the existence of $[X_i+X_j]$ and, hence, of $[X_i,X_j]$; see, e.g., the discussion in \cite{Schied_Takagi}.

 We can now state our It\^o formula for functionals in $\mathbb C^{1,2}_b(U,S)$, which slightly extends the one from \cite{Dupire,CF}. 
 
\begin{thm} \label{cvfcont} Let us fix a path $X\in C([0,T],U)$ with continuous quadratic variation, a path $A\in\text{\sl CBV}([0,T],S)$, and  a functional $F\in \mathbb{C}^{1,2}_b(U,S)$.
For $n\in\mathbb N$, define the approximating path $X^n\in D([0,T],U)$ by
 \begin{equation} 
X^n(t):=\sum_{s\in \mathbb{T}_n}X(s')\mathbbm{1}_{[s,s')}(t)+X(T)\Ind{\{T\}}(t),\qquad 0\leq t\leq T,
\label{Xn1}\end{equation}
and let $X^{n, s-}:=\lim_{r\uparrow s}X^{n, r}$. Then  the \emph{pathwise It\^o integral} along $\left(\mathbb{T}_n\right)$, \begin{equation} \label{Itointcont}
\int_0^T \nabla_XF(s, X, A)\ud X(s):=\lim_{n\uparrow\infty} \sum_{s\in \mathbb{T}_n}\nabla_XF(s,X^{n,s-}, A)\cdot\left(X(s')-X(s)\right),
\end{equation}
 exists and, with $A_0(t)=t$,
 \begin{equation}\label{Ito formula}
\begin{split} 
F(T,X ,A )-F(0,X ,A )&=\int_0^T \nabla_XF(s,X , A )\ud X(s)+\sum_{i=0}^m\int_0^T\mathscr{D}_iF(s,X,A)\ud A_i( s)\\
&\qquad+ \frac{1}{2}\sum_{i,j=1}^d\int_0^T\partial_{ij} F(s, X  ,A)\ud [X_i,X_j](s),
\end{split}
 \end{equation}
 where the two rightmost integrals are Lebesgue--Stieltjes integrals.
\end{thm}

\subsection{Associativity of the functional It\^o integral}\label{associativity section}

\blue{If $X\in C([0,T],U)$ admits the continuous quadratic variation $[X]$, then} Theorem~\ref{cvfcont} allows us in particular to define the  pathwise functional It\^o integral 
\begin{equation}\label{xi ito int eq}
\int_0^t \xi(s, X)\ud X(s):=\lim_{n\uparrow\infty} \sum_{s\in \mathbb{T}_n,\, s\le t}\xi(s,X^{n,s-})\cdot\left(X(s')-X(s)\right)
\end{equation}
for any functional $\xi:[0,T]\times D([0,T],U)\to\mathbb R^ d$ that is of the form
$$\xi(s,\blue{\wt X})=\nabla_XF(s,\blue{\wt X},A),\qquad \blue{\wt X\in D([0,T],U),}$$
for some $F\in\mathbb{C}^{1,2}_b(U,S)$ and $A\in\CBV([0,T],S)$. In contrast to standard pathwise It\^o calculus as in \cite{Ito_F,Sondermann,Schied13}, however, it is not clear \emph{a priori} whether the It\^ o integral in \eqref{xi ito int eq} is a continuous function of the upper integration bound~$t$. Lemma~\ref{regularity lemma} (c) states that this continuity condition will be guaranteed if $F$ satisfies the following additional regularity condition.

\begin{definition}
A non-anticipative functional  $F$ on $[0,T]\times D([0,T],U)\times\CBV([0,T],S)$ is called \emph{continuous in $X$  locally uniformly in $t$,}  if for all $\eps>0$ and $(t,X,A) \in [0,T]\times D([0,T],U)\times\text{\sl CBV}([0,T],S)$ there is some $\eta>0$ such that 
$$|F(u,X,A)-F(u,Y,A)|<\eps
$$
for all $(u,Y) \in [0,T]\times D([0,T],U)$ for which
$\|X -Y \|_\infty<\eta$ and $  |t-u|<\eta$. With $\mathbb C^{1,2}_c(U,S)$ we denote the class of all functionals $F\in \mathbb C^{1,2}_b(U,S)$ that are continuous in $X$  locally uniformly in $t$ and boundedness preserving.
\end{definition}

 As discussed in \cite{Ananova}, the main difference to standard pathwise It\^o calculus as in \cite{Ito_F,Sondermann,Schied13} is that the sums on the right-hand side of \eqref{xi ito int eq}
 are not ordinary Riemann sums but based on the approximations $X^n$ of $X$. See \cite[Theorem 3.2]{Ananova} for sufficient conditions under which these sums can be replaced by ordinary Riemann sums. 
 Here we will work with the following notion of an admissible integrand, which was suggested in \cite{Ananova}.
 
 \begin{definition}\label{locadmcont} A functional $\xi:[0,T]\times D([0,T],U)\to\mathbb R^d$ is called an \emph{admissible functional integrand} if there exist $m\in\mathbb N$, an open set $S\subset\mathbb R^m$, $A\in\CBV([0,T],S)$, and $F\in\mathbb{C}^{1,2}_c(U,S)$  such that $\xi(t,X)=\nabla_XF(t,X,A)$. If, moreover, $X\in C([0,T],U)$ is a path  with continuous quadratic variation, then  the It\^o integral of $\xi(t,X)$ against $X$ will be defined through \eqref{xi ito int eq}.
 \end{definition}

 Let $\xi_{(1)},\dots,\xi_{(\nu)}:[0,T]\times D([0,T],U)\to\mathbb R$ be admissible  functional  integrands.  For $X\in C([0,T],U)$ with continuous quadratic variation $[X]$ we define $Y=\left(Y_{1},\dots,Y_{\nu}\right)\in C([0,T],\mathbb R^\nu)$  through 
\begin{equation} 
Y_{\ell}(t):=\int_0^t\xi_{(\ell)}(s,X)\ud X(s),\quad \ell=1,\dots, \nu.
\label{Yellcont}\end{equation}
Then  
results from standard \blue{pathwise} It\^o calculus such as \cite[Proposition 2.3.3]{Sondermann} or \cite[Proposition 12]{Schied13} suggest that the continuous path $Y$
should admit the continuous quadratic variation
\begin{align} 
 [ Y_{k}, Y_{\ell} ](t)&=\sum_{i,j=1}^d\int_0^t\xi_{(k),i}(s,X)\xi_{(\ell),j}(s,X)\ud [ X_i,X_j](s),\qquad k,\ell=1,\dots,\nu.
\label{Yellcontcov}
\end{align} 
In functional It\^ o calculus, however, this is property is not immediately clear.  Ananova and Cont \cite[Theorem 2.1]{Ananova} recently gave sufficient conditions on $F$ and $X$ under which  \eqref{Yellcontcov} is satisfied. These conditions consist  mainly in a Lipschitz condition on $F$ with respect to $\|\cdot\|_\infty$ and the H\"older continuity of  $X$ with an exponent $\alpha>(\sqrt 3-1)/2$. \blue{In part (c) of the following theorem,} we will assume the validity of \eqref{Yellcontcov}  instead of imposing the conditions from \cite[Theorem 2.1]{Ananova}. 

Now we can state our associativity formula for the functional pathwise It\^o integral. It extends \cite[Theorem 13]{Schied13}, where an associativity formula for standard pathwise It\^o calculus was proved. Note that our current proof strategy can also be used to simplify the proof given in \cite{Schied13}, as we work here without the approximation arguments employed in \cite{Schied13}.

\begin{thm} \label{assoccont} Let $\xi_{(1)},\dots,\xi_{(\nu)}:[0,T]\times D([0,T],U)\to\mathbb R$ be admissible  functional  integrands.  For $X\in C([0,T],U)$ with continuous quadratic variation $[X]$ we define $Y=\left(Y_{1},\dots,Y_{\nu}\right)\in C([0,T],\mathbb R^\nu)$  through \eqref{Yellcont}.
Then the following assertions hold:
\begin{enumerate}
\item There exist $q\in\mathbb N$, an open set $\widetilde S\subset\mathbb R^q$, $\widetilde A\in \CBV([0,T],\widetilde S)$, and $\widetilde F_1,\dots,\widetilde F_\nu\in\mathbb C^{1,2}_c(U,\widetilde S)$ such that 
\begin{equation}\label{Yca}
Y(t)=\widetilde F(t,X,\widetilde A):=\big(\widetilde F_1(t,X,\widetilde A),\dots,\widetilde F_\nu(t,X,\widetilde A)\big)
\end{equation}
\item If $\eta:[0,T]\times C([0,T],\mathbb R^\nu)\to\mathbb R^\nu$ is an admissible  functional  integrand and $\widetilde F$ is as in {\rm(a)}, then   $\zeta:[0,T]\times  C([0,T],U)\to\mathbb R^d$, defined through
$$\zeta(t,\widetilde X):=\sum_{\ell=1}^\nu \eta_\ell\big(t,\widetilde F(\cdot,\widetilde X,\widetilde A)\big)\xi_{(\ell)}(t,\widetilde X),\qquad \widetilde X\in D([0,T],U),$$
is well defined and an admissible  functional  integrand.
\item If $Y$ admits the continuous quadratic variation \eqref{Yellcontcov} and $\zeta$ is as in {\rm(b)}, then the following associativity formula holds:
\begin{align} 
\int_0^T \eta (s,Y)\ud  Y(s) &= \int_0^T\zeta(t,X)\ud X(t)=\int_0^T\sum_{\ell=1}^\nu \eta_\ell(s,Y)\xi_{(\ell)}(s,X)\ud  X(s).\label{assco}
 \end{align}
\end{enumerate}
\end{thm}

$ $

The preceding theorem yields the following corollary.

$ $

\begin{cor}\label{ass cor}Let $\xi_{(1)},\dots,\xi_{(\nu)}:[0,T]\times D([0,T],U)\to\mathbb R$ be admissible  functional  integrands and $X\in C([0,T],U)$ with continuous quadratic variation $[X]$. Let, moreover, $S\subset\mathbb R^m$ be open, $A\in\CBV([0,T],S)$, functionals  $F_1,\dots, F_\nu\in\mathbb C^{1,2}_c(U, S)$ and
$$
Y(t):= \big(  F_1(t,X, A),\dots,  F_\nu(t,X, A)\big)
$$ 
be such that it admits the continuous quadratic variation
$$[Y_i,Y_j](t)=\sum_{k,\ell=1}^d\int_0^t\partial_k F_i(s,X,A)\partial_\ell F_j(s,X,A)\ud[X_k,X_\ell](s),\qquad i,j=1,\dots,\nu.
$$
Then, if $\eta:[0,T]\times C([0,T],\mathbb R^\nu)\to\mathbb R^\nu$ is an admissible  functional  integrand, 
\begin{equation}\label{ass cor func int eq}
\sum_{\ell=1}^\nu\eta_\ell\big(t,F(\cdot,\blue{\wt X},A)\big)\nabla_XF_\ell(t,\blue{\wt X},A),\qquad \blue{\wt X\in D([0,T],U),}
\end{equation}
is an admissible functional integrand and
\begin{align*}
\int_0^T \eta (s,Y)\ud  Y(s)&=\int_0^t\sum_{\ell=1}^\nu\eta_\ell\big(s,F(\cdot,X,A)\big)\nabla_XF_\ell(t,X,A)\ud X\\
&\qquad +\sum_{\ell=1}^\nu\bigg(\sum_{i=0}^m\int_0^t\eta_\ell\big(s,F(\cdot,X,A)\big)\mathscr D_iF_\ell(s,X,A)\ud A_i(s)\\
&\qquad +\frac12\sum_{i,j=1}^d\int_0^t\eta_\ell\big(s,F(\cdot,X,A)\big)\partial_iF_\ell(s,X,A)\partial_jF_\ell(s,X,A)\ud[X_i,X_j](s)\bigg)
\end{align*}
\end{cor}

To provide a quick,  easy, and relevant example for the possible applications of Theorem~\ref{assoccont} and Corollary~\ref{ass cor}, we consider a situation that appears in the context of our companion paper~\cite{SPVol15}.  

\begin{ex}Let $\eta:[0,T]\times D([0,T],\mathbb R)\to\mathbb R$ be an admissible functional integrand and $X\in C([0,T],(0,\infty))$ be such that $X$ has the continuous quadratic variation $[X]$. Then Corollary~\ref{ass cor} yields that ${\eta(s,\log \blue{\wt X})}/{\blue{\wt X}(s)}$, \blue{$\wt X\in D([0,T],(0,\infty))$,} is an admissible functional integrand and that we have the following change of variables formula:
$$\int_0^t\eta(s,\log X)\ud\log X(s)=\int_0^t\frac{\eta(s,\log X)}{X(s)}\ud X(s)-\frac12\int_0^t\frac{\eta(s,\log X)}{(X(s))^2}\ud[X](s).
$$
In standard stochastic calculus, the preceding identity is normally obtained from the It\^o formula for $\log X$ in conjunction with the Kunita--Watanabe characterization of the stochastic integral. Since the Kunita--Watanabe characterization is not available in our present   context of pathwise functional It\^o calculus, the present example illustrates \blue{\sout{for}} the need \blue{for} the associativity formulas from Theorem~\ref{assoccont} and Corollary~\ref{ass cor}.
\end{ex}


\begin{rem}The definition \eqref{xi ito int eq}, which involves the approximations $X^n$, is not the only conceivable definition of the pathwise functional It\^ o integral. As a matter of fact, \cite[Theorem 3.2]{Ananova} states conditions under which this It\^o integral is equal to the following limit of proper Riemann sums,
$$\int_0^t \xi(s, X)\ud X(s)=\lim_{n\uparrow\infty} \sum_{s\in \mathbb{T}_n,\, s\le t}\xi(s,X)\cdot\left(X(s')-X(s)\right).
$$
In view of this fact, it is important to note that our proof of Theorem~\ref{assoccont}  (c) is only based on the validity of the It\^o formula \eqref{Ito formula}. It does not involve the particular approximation  \eqref{xi ito int eq}. Also in this respect, our current proof improves the one of \cite[Theorem 13]{Schied13}.
\end{rem}

\section{Proofs}\label{proofs section}

\subsection{Proof for Section~\ref{subsect 2.2}}

\begin{proof}[Proof of Theorem~\ref{cvfcont}]
The proof is a variant of the ones in \cite{Dupire} and~\cite[Theorem 3]{CF}.
For $s\in \mathbb{T}_n$, consider the following decomposition of increments into ``horizontal'' and ``vertical'' terms:
\begin{align} 
\lefteqn{F(s',X^{n,s'-},A)-F(s,X^{n,s-},A)}\\
&=F(s',X^{n,s'-},A)-F(s,X^{n,s},A)+F(s,X^{n,s},A)-F(s,X^{n,s-},A).
\label{decompcont}\end{align}
Since $X^{n,s'-}=X^{n,s}$, we can re-write the first difference on the right as 
\begin{align}
F(s',X^{n,s'-},A)-F(s,X^{n,s},A)&=\sum_{i=0}^d \int_s^{s'}\mathscr{D}_iF(r,X^{n,s},A)\,\ud A_i(r) .
\label{decompcont1}\end{align}
For the second term on the right-hand side of \eqref{decompcont}, we have
\begin{equation} 
F(s,X^{n,s},A)-F(s,X^{n,s-},A)=F\big(s,X^{n,s-}+\delta X^n_s\Ind{[s,T]},A\big)-F(s,X^{n,s-},A),
\label{decompcont2}\end{equation}
where
$\delta X_{s}^n:=X(s')-X(s)$.
Hence, a second-order Taylor expansion  yields
\begin{align} 
F(s,X^{n,s},A)-F(s,X^{n,s-},A)&=
\nabla_XF(s,X^{n,s-}, A)\cdot\delta X_{s}^n\notag\\
&\quad+\: \frac{1}{2} \sum_{i,j=1}^d\partial_{ij} F(s,X^{n,s-}, A)(\delta X_{s}^n)_i(\delta X_{s}^n)_j  +r_n(s)
\label{phiTaylor}\end{align}
where
for some numbers $\theta_{ij,s}\in[0,1],$ 
 \begin{align}\label{remainder term in Ito formula}
r_n(s):=\frac{1}{2}\sum_{i,j=1}^d\bigg(\partial_{ij} F\Big(s,X^{n,s-}+\theta_{ij,s}\delta X_s^n\Ind{[s,T]}, A\Big) - \partial_{ij} F(s,X^{n,s-}, A)\bigg)(\delta X_{s}^n)_i(\delta X_{s}^n)_j.\end{align} 

We now sum over $s\in \mathbb{T}_n$ and investigate the limit as $n\uparrow\infty$. 

First, the left-hand side of \eqref{decompcont} sums up to $F(T,X^{n,T-},A)-F(0,X^{0},A)$, which    converges to $F(T,X^{T-},A)-F(0,X,A)$, due to the continuity of $F$ at time $T$. Since $X$ is continuous, this limit is equal  to $F(T,X,A)-F(0,X,A).$

Second, for the first term on the right-hand side of \eqref{decompcont}, we get  with \eqref{decompcont1} that
\begin{equation} 
\begin{split}
 \sum_{s\in \mathbb{T}_n} \Big(F(s',X^{n,s'-}A)-F(s,X^{n,s},A)\Big)=\sum_{i=0}^d \int_0^T\mathscr{D}_iF(r,X^{n,s(r)},A)\,\ud A_i(r)
\end{split}
\label{decompcont1_2}\end{equation}
where $s(T):=T$ and  $s(r)$ is defined for $r<T$ as that  $s\in\mathbb T_n$ for which  $u\in[s,s')$.
Since the horizontal derivatives are continuous at fixed times, we get $\mathscr{D}_iF(r,X^{n,s(r)},A)\to \mathscr{D}_iF(r,X, A)$ for each $r$ as $n\uparrow\infty$.  Moreover, the integrands $\mathscr{D}_iF(r,X^{n,s(r)},A)$ are bounded uniformly in $r$ and $n$ by Remark~\ref{boundedness preserving rem}. Therefore, \eqref{decompcont1_2} converges to 
$$ \sum_{i=0}^m\int_0^T\mathscr{D}_iF(s,X,A)\ud A_i(s).
$$

Third, for the second term on the right-hand side of \eqref{decompcont}, we have  
\begin{equation}\label{decompcont2_1}
\begin{split}
\lefteqn{\sum_{s\in\mathbb{T}_n} \Big(F(s,X^{n,s},A)-F(s,X^{n,s-},A)\Big)}\\
&=\sum_{s\in \mathbb{T}_n}\nabla_XF(s,X^{n,s-}, A)\cdot\delta X_{s}^n +\: \frac{1}{2} \sum_{s\in \mathbb{T}_n}\sum_{i,j=1}^d \partial_{ij} F(s,X^{n,s-}, A)(\delta X_{s}^n)_i(\delta X_{s}^n)_j +\sum_{s\in \mathbb{T}_n}r_n(s). 
\end{split}
\end{equation}
Let $\varphi_n(u):=\nabla^2_X F(s(u),X^{n,s(u)-}, A)$ where $s(u)$ is as above. Then the functions $\varphi_n$ are uniformly bounded according to Remark~\ref{boundedness preserving rem}. Next, the left continuity of $\nabla^2_XF$ implies that $\varphi_n(u)\to \nabla^2_X F(u,X, A)$ as $n\uparrow\infty$. Furthermore,  \cite[ Proposition 1]{CF} and the left continuity of $\nabla^2_XF$ imply that both $\varphi_n(u)$ and $\nabla^2_X F(u,X, A)$ are left-continuous functions of $u$. Therefore,   \cite[ Lemma 12]{CF} yields that 
\begin{align} 
\sum_{s\in \mathbb{T}_n}  \sum_{i,j=1}^d \partial_{ij} F(s,X^{n,s-}, A)(\delta X_{s}^n)_i(\delta X_{s}^n)_j\longrightarrow\sum_{i,j=1}^d\int_0^T\partial_{ij} F(s,X, A)\ud [ X_i,X_j](s). 
\label{decompcont2_2}\end{align}

Fourth, we deal with the remainder terms $r_s^n$. To this end, we let
$$\rho^n_{ij}(u):=\bigg(\partial_{ij} F\Big(s(u),X^{n,s(u)-}+\theta_{ij,s(u)}\delta X_{s(u)}^n\Ind{[s(u),T]}, A\Big) - \partial_{ij} F(s(u),X^{n,s(u)-}, A)\bigg),
$$
where $s(u)$ is as above.  The left continuity of $\nabla_X^2F$, the fact that $|\theta_{ij,n}\delta X_s^n|\le |\delta X_s^n|\to 0$, and  Remark~\ref{boundedness preserving rem} imply that $\rho^n_{ij}(u)\to0$ boundedly as $n\uparrow\infty$. Moreover, \cite[ Proposition 1]{CF} implies again the left continuity of $\rho^n_{ij}(u)$. Thus, an application of \cite[ Lemma 12]{CF} yields that 
$$\sum_{s\in\mathbb T_n}r_n(s)= \sum_{s\in\mathbb T_n}\rho^n_{ij}(s)(\delta X_n)_i(\delta X_n)_j\longrightarrow0\qquad\text{as $n\uparrow\infty$.}
$$

Finally, putting everything together implies that all terms converge. Therefore, the limit
\[
\lim_{n\to\infty} \sum_{s\in\mathbb{T}_n}\nabla_XF(s,X^{n,s-}, A)\cdot\left(X(s')-X(s)\right)
\]
must also exist, and the theorem follows.\end{proof}

\subsection{Proofs for Section~\ref{associativity section}}

We start with a simple consequence of Theorem~\ref{cvfcont}.

\begin{lemma}\label{Ito Stieltjes lemma} Let $\xi:[0,T]\times D([0,T],U)\to \mathbb R^d$ be an admissible functional integrand and   $B\in\CBV([0,T],U)$. Then the It\^o integral 
\begin{equation}\label{CBV Ito eq}
\int_0^T\xi(s,B)\ud B(s)=\lim_{n\uparrow\infty} \sum_{s\in \mathbb{T}_n}\xi(s,B^{n,s-})\cdot\left(B(s')-B(s)\right)
\end{equation}
exists and is equal to the Lebesgue--Stieltjes integral of $s\mapsto\xi(s, B)$ with respect to $\ud B(s)$.
\end{lemma}

\begin{proof}Since $B$ admits the continuous quadratic variation $[B_i,B_j]=0$  by \cite[Proposition 2.2.2]{Sondermann}, we may apply Theorem~\ref{cvfcont} so as to obtain the existence of the It\^o integral. Let 
$\mu_n=\sum_{t\in\mathbb T_n}(B(t')-B(t))\delta_t$. Then $\mu_n$ converges vaguely to $\ud B$, since Riemann sums for continuous integrands converge to the Stieltjes integral.  \blue{Moreover, $s\mapsto \xi(s,B^{n,s-})$ is left-continuous and we have $\xi(s,B^{n,s-})\to \xi(s,B)$, due to the left continuity of the functional $\xi$. Also, $\xi(s,B^{n,s-})$ is uniformly bounded in $s$ and $n$by Remark~\ref{boundedness preserving rem}.} Therefore, \cite[Lemma 12]{CF} yields that the right-hand side of \eqref{CBV Ito eq} converges to the  Lebesgue--Stieltjes integral of $s\mapsto \blue{\xi(s, B)}$ with respect to $\ud B(s)$.
\end{proof}

The following lemma uses some ideas from \cite[Proposition 1]{CF}.

\begin{lemma}\label{regularity lemma}
  Let $F\in \mathbb{C}^{1,2}_c(U,S)$.
\begin{enumerate}
\item For all $\eps>0$ and $(X,A) \in  D([0,T],U)\times\text{\sl CBV}([0,T],S)$ there is some $\eta>0$ such that $\|X -Y \|_\infty<\eta$ implies that 
$$\sup_{t\in[0,T]}|F(t,X,A)-F(t,Y,A)|<\eps.
$$
\item For all $X\in D([0,T],U)$ and $A\in\CBV ([0,T],S)$, the function $t\mapsto F(t,X,A)$ is c\`adl\`ag, and its left-hand limit at $t\in(0,T]$ is given by
\begin{equation}\label{left-hand limit F}
\lim_{s\uparrow t}F(s,X,A)=F(t,X^{t-},A).
\end{equation}
\item If  $X\in C([0,T],U)$ admits the continuous quadratic variation $[X]$, then  the pathwise It\^o integral $\int_0^t \nabla_X F(s,X,A)\ud X(s)$ is a continuous function of $t$.

\end{enumerate}
  \end{lemma}
  
  \begin{proof}[Proof of Lemma~\ref{regularity lemma}] 
 For part (a), let $\eps>0$ and $(X,A)$ be given. For every $t\in[0,T]$ there exists some $\eta_t>0$ such that $|F(u,X,A)-F(u,Y,A)|<\eps$ if $\|X -Y \|_\infty<\eta_t$ and $  |t-u|<\eta_t$, \blue{because $F$ is continuous in $X$, locally uniformly in $t$}. The collection of all $\eta_t$-balls, $(t-\eta_t,t+\eta_t)$, with $t\in[0,T]$ covers $[0,T]$. Hence there exists a finite subcover centered at $t_1,\dots, t_n$. Thus, $\eta:=\min\{\eta_{t_1},\dots,\eta_{t_n}\}$ is as desired.
 
For (b), note first that \eqref{left-hand limit F}  follows immediately from the left continuity of $F$ together with the fact that 
$\|X^s-X^{t-}\|_\infty\to 0$ as $s\uparrow t$.
Now we show  right continuity. To this end, we write for $h>0$,
$$F(t+h,X,A)-F(t,X,A)=F(t+h,X,A)-F(t+h,X^t,A)+F(t+h,X^t,A)-F(t,X,A).
$$
Now we investigate the limit $h\downarrow0$. Since $\|X^{t+h}-X^t\|_\infty\to 0$ by the right-continuity of $X$, we get  $F(t+h,X,A)-F(t+h,X^t,A)\to0$  from the fact that $F$ is continuous \blue{in $X$,} locally uniform in $t$. Next, since  the horizontal derivatives are boundedness preserving, we get with $A_0(s):=s$,
\begin{equation}
F(t+h, X^{t}, A)-F(t, X, A)=\sum_{i=0}^m\int_t^{t+h} \mathscr{D}_i F(r, X , A )\ud A_i(r)\longrightarrow 0 .
\label{Fstetig2}\end{equation}
Putting these facts together establishes the right continuity of $t\mapsto F(t,X,A)$. 

Now we prove (c). Since   $A_i(s)$ and $[X_i,X_j](s)$ are continuous functions of bounded variation, the corresponding Lebesgue--Stieltjes integrals are continuous functions of their upper integration limit. Therefore, the It\^o formula \eqref{Ito formula} implies that $\int_0^t \nabla_X F(s,X,A)\ud X(s)$ must be continuous in $t$ as soon as $t\mapsto F(t,X,A)$ is continuous. But this continuity follows from part (b), \eqref{left-hand limit F}, and the continuity of $X$.
  \end{proof}

For the proof of Theorem~\ref{assoccont} we need the following auxiliary lemmas. The first one is a product rule for vertical derivatives, the second one a chain rule for both vertical and horizontal derivatives. Both extend statements from~\cite{Dupire}.

\begin{lemma}\label{product rule} Let $F$ and $G$ be two non-anticipative functionals that are boundedness-preserving, left-continuous, and vertically differentiable with left-continuous and boundedness-preserving vertical derivatives, $\nabla_X F$ and $\nabla_X G$.  Then the product $FG$ is again a non-anticipative vertically differentiable functional such that $FG$ and $\nabla_X(FG)$ are left-continuous and boundedness preserving. Moreover,
  \begin{equation}\label{producteq}
 \nabla_X (FG)=G\nabla_XF  + F \nabla_XG .
 \end{equation}
\end{lemma}
\begin{proof} If $F$ and $G$ are boundedness preserving and left-continuous, then so is their product $FG$. Next, the product rule \eqref{producteq} follows immediately from the definition of the vertical derivative.
Moreover, all functionals appearing on the right-hand side of \eqref{producteq} are boundedness preserving and  left-continuous as the products of such functionals. 
\end{proof}

\begin{lemma}\label{chain rule}Let  $U\subset\mathbb{R}^d$, $V\subset\mathbb{R}^\nu$, $S\subset\mathbb R^m$, and $W\subset\mathbb{R}^p$ be open sets. Let  $F_1,\dots, F_\nu\in \mathbb{C}^{1,2}_c(U,S)$ be  such that 
$$F(t,X,A):=(F_1(t,X,A),\dots,F_\nu(t,X,A)) \in V$$
for all $(t,X,A)\in [0,T]\times D([0,T],U)\times \text{\sl CBV}([0,T],S)$. Suppose moreover that $G\in \mathbb{C}^{1,2}_c(V,W)$.  Then, for $ t\in[0,T]$, $ X\in D([0,T],U) $, $ A\in \text{\sl CBV}([0,T],S)$, and $B\in \text{\sl CBV}([0,T],W) $, the functional 
$$H(t,X,(A,B)):=G(t,F(\cdot ,X,A),B),
$$
is well defined and belongs to the class $\mathbb{C}^{1,2}_c(U,S\times W)$. Its vertical and horizontal derivatives are given as follows:
\begin{align}
\partial_iH(t,X,(A,B))&=\sum_{\ell=1}^\nu \partial_\ell G(t,F(\cdot ,X,A),B)\partial_iF_\ell(t,X,A),\qquad i=1,\dots, d,\label{vertical chain rule}\\
\partial_{ji}H(t,X,(A,B))&=\sum_{\ell=1}^\nu\bigg(\sum_{k=1}^\nu\partial_{\ell k} G(,F(\cdot,X,A),B)\partial_iF_\ell(t,X,A)\partial_jF_k(t,X,A)\nonumber\\
&\qquad\qquad+\partial_\ell G(t,F(\cdot,X,A),B)\partial_{ji}F_\ell(t,X,A)\bigg)\label{vertical chain rule order 2}\\
 \mathscr{D}_0H(t,X,(A,B))&=\mathscr{D}_0G(t,F(\cdot,X,A),B)+\sum_{\ell=1}^\nu\partial_\ell G(t,F(\cdot,X,A),B)\mathscr{D}_0F_\ell(t,X,A)\label{chain rule horizontal derivatives 0}\\
\mathscr{D}_iH(t,X,(A,B))&=\sum_{\ell=1}^\nu\partial_\ell G(t,F(\cdot,X,A),B)\mathscr{D}_iF_\ell(t,X,A),\quad i=1,\dots,m,\label{chain rule horizontal derivatives 1}\\
\mathscr{D}_jH(t,X,(A,B))&=\mathscr{D}_{j-m}G(t,F(\cdot,X,A),B),\quad j=m+1,\dots, m+p.\label{chain rule horizontal derivatives 2}
\end{align}

\end{lemma}

\begin{proof}First note that $H$ is well defined since  $F(\cdot,X,A)\in D([0,T],V)$ by Lemma~\ref{regularity lemma} (b). Parts (a) and (b) of that lemma actually state that $X\mapsto F(\cdot ,X,A)$ is a continuous map from $D([0,T],U)$ into $D([0,T],V)$ with respect to the corresponding supremum norms. We thus get immediately that $H$ is boundedness-preserving, left-continuous, and continuous \blue{in $X$,} locally uniformly in $t$. Similarly, we see that the functionals on the right-hand sides of \eqref{chain rule horizontal derivatives 0}, \eqref{chain rule horizontal derivatives 1}, and \eqref{chain rule horizontal derivatives 2} are boundedness-preserving and continuous at fixed times and that the expressions on the right-hand sides of \eqref{vertical chain rule} and \eqref{vertical chain rule order 2} are boundedness-preserving and left-continuous.

Next, we prove our assertions concerning vertical differentiability. To this end, let us fix $t\in[0,T]$ and define $v(h):=F(t,X+he_i\Ind{[t,T]},A)-F(t,X,A)$ for $h\in\mathbb R$. Then, since $F$ and $G$ are non-anticipative,
\begin{align*}
H\big(t,X+he_i\Ind{[t,T]},(A,B)\big)&=G\Big(t, F(\cdot,X,A)+\big(F(\cdot,X+he_i\Ind{[t,T]},A)-F(\cdot,X,A)\big)\Ind{[t,T]},B\Big)\\
&=G\big(t,F(\cdot,X,A)+v(h)\Ind{[t,T]},A),B\big).
\end{align*}
Thus, the chain rule from standard calculus implies the vertical differentiability of   $H$ and the vertical chain rule \eqref{vertical chain rule}. Iterating this argument and combining it with the product rule from Lemma~\ref{product rule} then gives \eqref{vertical chain rule order 2}.

Next, we turn to our assertions concerning the horizontal differentiability. To this end, we fix $s\in [0,T)$ and define for $t\in[0,T]$ the $\nu$-dimensional path
$C(t):=F(t,X^s,A)$. The horizontal differentiability gives that the $\ell^{\text{th}}$ component of $C$ satisfies for $t\ge s$,
$$C_\ell(t)=F_\ell(s,X,A)+\sum_{i=0}^{m}\int_s^t\mathscr{D}_iF_\ell(r,X^s,A)\ud A_i( r),
$$
where again $A_0(r)=r$. Clearly, the restriction of $C$ to $[s,T]$ belongs to $\CBV([s,T],V)$. In particular, the quadratic variation $[C]$ of $C$ exists  and vanishes identically if taken from time $s$ onward \cite[Proposition 2.2.2]{Sondermann}.  Therefore, Theorem~\ref{cvfcont} yields that for $t\in[s,T]$,
\begin{align*}
\lefteqn{H(t,X^s,(A,B))-H(s,X,(A,B))}\\
&=G(t,F(\cdot,X^s,A),B)-G(s,F(\cdot,X^s,A),B)\\
&=G(t,C,B)-G(s,C,B)\\
&=\int_s^t\nabla_YG(r,C,B)\ud C(r)+\sum_{j=0}^n\int_s^t\mathscr{D}_jG(r,C,B) \ud B_j( r),
\end{align*}
where $B_0(r)=r$. Here, $\int_s^t\nabla_YG(r,C,B)\ud C(r)$ is a Lebesgue--Stieltjes integral by Lemma~\ref{Ito Stieltjes lemma}. The associativity of the Lebesgue--Stieltjes  integral, which follows from the Radon--Nikodym theorem,  yields that
\begin{align*}
\int_s^t\nabla_YG(r,C,B)\ud C(r)&=\sum_{i=0}^m\int_s^t\sum_{\ell=1}^\nu\partial_\ell G(r,C,B)\mathscr{D}_iF_\ell(r,X^s,A)\ud A_i( r).
\end{align*}
It thus follows that $H$ admits the horizontal derivatives \eqref{chain rule horizontal derivatives 0}, \eqref{chain rule horizontal derivatives 1}, and \eqref{chain rule horizontal derivatives 2}.
\end{proof}

\begin{proof} [Proof of Theorem~\ref{assoccont}]
(a) Every $\xi_{(\ell)}$ is of the  form 
$\xi_{(\ell)}(t,X)=\nabla_X F_\ell (t, X, A_{(\ell)}),\;t\in[0,T],$ for $F_\ell \in\mathbb C^{1,2}_c(U,S_\ell)$ and some $A_{(\ell)}\in \CBV([0,T],S_\ell)$. Clearly, there is no loss of generality if we assume that $S_1=\cdots=S_\nu=S$ and $A_{(1)}=\cdots=A_{(\nu)}=A$ for some open set $S\subset\mathbb{R}^{m}$  and some $A\in \CBV([0,T],S)$; for instance, we can always take $S=S_1\times\cdots\times S_\nu$ and $A=(A_{(1)},\dots,A_{(\nu)})$.  Then, Theorem~\ref{cvfcont} yields
\begin{align*} 
&  F_\ell ({t},X ,A )-F_\ell (0,X,A)=\int_{0}^{t }\nabla_XF_\ell (s,X, A )\ud X(s)\\
&\quad+\sum_{i=0}^{m}
\int_{0}^{t}\mathscr{D}_i F_\ell (s, X , A  )\ud A_{i}(s)+\:\frac{1}{2}\sum_{i,j=1}^d\int_{0}^{t}\partial_{ij} F_\ell (s, X , A  )\ud [ X_i,X_j](s) ,
\end{align*}
where $A_{0}( s)=s$ as usual. 
Introducing
\begin{align} 
& A_{m +\ell}(t):= \sum_{i=0}^{m}
\int_{0}^{t}\mathscr{D}_i F_\ell (s, X , A  )\ud A_{i}(s)+\frac{1}{2}\sum_{i,j=1}^d\int_{0}^{t}\partial_{ij} F_\ell (s, X , A  )\ud [X_i,X_j](s)
\label{Aellkmlk+1}\end{align}
we set $\widetilde{A}:= \left(A_{1},\dots,A_{m},A_{m+1},\dots,A_{m+\nu}\right) $. With $\widetilde S:= S\times\mathbb R^\nu$, we then have $\widetilde{A}\in\CBV([0,T],\widetilde S)$ 
due to  standard properties of Lebesgue--Stieltjes   integrals. Moreover, we can  write
\begin{align}
Y_{(\ell)}(t)&=F_\ell (t,X,A )-F_\ell (0,X,A )- A_{m+\ell}(t)=:\widetilde F_\ell (t,X,\widetilde{A}).
\label{Yellca}\end{align}
The functional $\widetilde F_\ell $ is  clearly of class $\mathbb{C}^{1,2}_c(U,\widetilde S)$ with 
\begin{equation}\label{nabla tilde F = nabla F}
\begin{split}
\nabla_X\widetilde F_\ell (t,X,\widetilde{A})&=\nabla_XF_\ell (t,X,A ),\\
\mathscr{D}_i\widetilde F_\ell (t,X,\widetilde{A})&= \mathscr{D}_i {F}_{\ell}(t,X, {A})\qquad\text{for $i=0,\dots, m$,}\\
\mathscr{D}_i\widetilde F_\ell (t,X,\widetilde{A})&= -\delta_{i\ell}\qquad\text{for $i=m+1,\dots, m+\nu$,}
\end{split}
\end{equation}
where $\delta_{i\ell}$ is the Kronecker delta.
Denoting 
\[
\widetilde{F} (t,X,\widetilde{A} ):=\left( \widetilde F_1(t,X,\widetilde{A} ),\dots,\widetilde{F}_{\nu}(t,X,\widetilde{A})\right) , 
\]
the identity \eqref{Yellca} yields \eqref{Yca}.

(b) The admissible functional integrand $\eta$  can  be written  as
$\eta(t,\cdot)=\nabla_YG(t,\cdot, B)$, for some $W\subset\mathbb{R}^p$, $G\in\mathbb C^{1,2}_c(V,W)$, and $B\in \text{\sl CBV}([0,T],W) $. 
 Lemma~\ref{chain rule} and \eqref{Yca} yield for $i=1,\dots, d$,
\begin{align}
\zeta(t,X)&=\sum_{\ell=1}^\nu \eta_\ell\big(t,\widetilde{F} (\cdot ,X,\widetilde{A} )\big)\xi_{(\ell),i}(t,X)\nonumber\\
&= \sum_{\ell=1}^\nu \partial_\ell G \big((t,\widetilde{F} (\cdot,X,\widetilde{A} ),B\big)\partial_i\widetilde F_\ell(t, X,\widetilde{A})\label{H gradient}\\
&= \partial_i H(t,X,(\widetilde{A},B)),\nonumber
\end{align}
where  
$$
H(t,X,(\widetilde{A},B)):=G\big(t,\widetilde{F}(\cdot,X ,\widetilde{A}), B\big) 
$$
  belongs to $\mathbb{C}^{1,2}_c(U,S\times\mathbb R^\nu\times W)$ by Lemma~\ref{chain rule}. 
 We hence infer that $\zeta$ is indeed well defined and an admissible  functional  integrand. This completes the proof of (b)

  (c) We now assume that $Y$ satisfies \eqref{Yellcontcov}. Applying It\^o's formula in the form of Theorem~\ref{cvfcont} to $G(t,Y,B)$ yields that
 \begin{equation}\label{G Ito}
 \begin{split}
 \int_0^T\eta(t,Y)\ud Y(t)&=G(T,Y,B)-G(0,Y,B)-\sum_{i=0}^p\int_0^T\mathscr{D}_iG(t,Y,B)\ud B_i(t)\\
 &\qquad-\frac12\sum_{k,\ell=1}^\nu\int_0^T\partial_{k\ell}G(t,Y,B)\ud[Y_k,Y_\ell](t).
  \end{split}
 \end{equation}
 Our goal is to identify the right-hand side of \eqref{G Ito} with $\int_0^T\sum_{\ell=1}^\nu \eta_\ell(s,Y)\xi_{(\ell)}(s,X)\ud  X(s)$. To this end, we will separately analyze each   term on the right-hand side of   \eqref{G Ito}. 
 
 First,
 \begin{equation}\label{term 1}
  \begin{split}G(T,Y,B)-G(0,Y,B)&=G\big(T,\widetilde{F}(\cdot,X ,\widetilde{A}), B\big)-G\big(0,\widetilde{F}(\cdot,X ,\widetilde{A}), B\big) \\
 &=H(T,X,(\widetilde{A},B))-H(0,X,(\widetilde{A},B))  \end{split} 
 \end{equation}

Second, the identity \eqref{chain rule horizontal derivatives 2} yields that 
\begin{equation}\label{term 2}
\begin{split}
\lefteqn{\sum_{i=0}^p\int_0^T\mathscr{D}_iG(t,Y,B)\ud B_i(t)}\\
&=\int_0^T\mathscr D_0G\big(t,\wt F(\cdot,X,\wt A),B\big)\ud t+\sum_{j=m+\nu+1}^{m+\nu+p}\int_0^T\mathscr{D}_jH(t,X,(\widetilde{A},B))\ud B_{j-m-\nu}(t).
\end{split}
\end{equation}

Third, we analyze the rightmost term in \eqref{G Ito}. Using our assumption  \eqref{Yellcontcov} on the quadratic variation of $Y$, the associativity of the Lebesgue--Stieltjes   integral, which follows from the Radon--Nikodym theorem,  the identities \eqref{nabla tilde F = nabla F} and \eqref{Yca}, we get
 \begin{align}
\lefteqn{\frac12\sum_{k,\ell=1}^\nu \int_0^T\partial_{k\ell}G(t,Y,B)\ud[Y_k,Y_\ell](t)}\nonumber\\
&=\frac12\sum_{k,\ell=1}^\nu\sum_{i,j=1}^d\int_0^T\partial_{k\ell}G(t,Y,B)\xi_{(k),i}(t,X)\xi_{(\ell),j}(t,X)\ud[X_i,X_j](t)\nonumber\\
&=\frac12\sum_{k,\ell=1}^\nu\sum_{i,j=1}^d\int_0^T\partial_{k\ell}G(t,Y,B) \partial_iF_k (t, X, A)\partial_jF_\ell (t,X,A)\ud[X_i,X_j](t)\nonumber\\
&=\frac12\sum_{i,j=1}^d\sum_{k,\ell=1}^\nu\int_0^T\partial_{k\ell}G\big(t,\widetilde{F}(\cdot,X,\widetilde{A}),B\big) \partial_i\widetilde{F}_{(k)}(t, X,\widetilde A)\partial_j\widetilde F_\ell(t,X,\widetilde A )\ud[X_i,X_j](t)\nonumber\\
&=\frac12\sum_{i,j=1}^d\int_0^T\partial_{ij}H(t,X,(\widetilde A,B))\ud[X_i,X_j](t)\nonumber\\
&\qquad-\frac12\sum_{\ell=1}^\nu\sum_{i,j=1}^d\int_0^T\partial_\ell G\big(t,\widetilde{F}(\cdot,X,\widetilde{A}),B\big)\partial_{ij}\widetilde F_\ell (t,X,\widetilde A )\ud[X_i,X_j](t).\nonumber
 \end{align}
 Next, our definition \eqref{Aellkmlk+1} of $A_{m+\ell}$ and  the associativity of the Lebesgue--Stieltjes   integral  imply that the latter terms satisfy
 \begin{align}
 \lefteqn{\frac12\sum_{\ell=1}^\nu\sum_{i,j=1}^d\int_0^T\partial_\ell G\big(t,\widetilde{F}(\cdot,X,\widetilde{A}),B\big)\partial_{ij}\widetilde F_\ell (t,X,\widetilde A )\ud[X_i,X_j](t)}\nonumber\\
 &=\sum_{\ell=1}^\nu\int_0^T\partial_\ell G\big(t,\widetilde{F}(\cdot,X,\widetilde{A}),B\big)\ud A_{m+\ell}(t)\label{fin var terms a1}\\
 &\qquad-\sum_{\ell=1}^\nu\sum_{i=0}^{m}\int_0^T\partial_\ell G\big(t,\widetilde{F}(\cdot,X,\widetilde{A}),B\big)\mathscr{D}_iF_\ell (t,X,A )\ud A_{i}(s)\nonumber
 \end{align}
The identity \eqref{nabla tilde F = nabla F} allows us to replace $\mathscr{D}_iF_\ell (t,X,A )$ by $\mathscr{D}_i\widetilde F_\ell (t,X,\widetilde A )$ in the rightmost term. Thus, using \eqref{chain rule horizontal derivatives 0} and \eqref{chain rule horizontal derivatives 1}, we get
\begin{equation}\label{Horizontal H terms}
\begin{split}
\lefteqn{\sum_{\ell=1}^\nu\sum_{i=0}^{m}\int_0^T\partial_\ell G\big(t,\widetilde{F}(\cdot,X,\widetilde{A}),B\big)\mathscr{D}_iF_\ell (t,X,A )\ud A_{i}(s)}\\
&=\sum_{i=0}^{m}\int_0^T\mathscr{D}_i H(t,X,(\widetilde A,B))\ud A_i(t)-\int_0^T\mathscr D_0G\big(t,\wt F(\cdot,X,\wt A),B\big)\ud t
\end{split}
\end{equation}
 For the first term on the right-hand side of \eqref{fin var terms a1}, the third identity in \eqref{nabla tilde F = nabla F}  gives
 $$\sum_{\ell=1}^\nu\int_0^T\partial_\ell G\big(t,\widetilde{F}(\cdot,X,\widetilde{A}),B\big)\ud A_{m+\ell}(t)=-\sum_{\ell=1}^\nu\int_0^T\mathscr{D}_{m+\ell}H(t,X,(\widetilde A,B))\ud A_{m+\ell}(t).
 $$
 Thus, the expression in \eqref{fin var terms a1} equals
 $$-\sum_{i=0}^{m+\nu}\int_0^T\mathscr{D}_i H(t,X,(\widetilde A,B))\ud A_i(t).
 $$
 We have thus shown that
 \begin{equation}\label{term 3}
 \begin{split}
 \lefteqn{\frac12\sum_{k,\ell=1}^\nu \int_0^T\partial_{k\ell}G(t,Y,B)\ud[Y_k,Y_\ell](t)}\\
 &=\frac12\sum_{i,j=1}^d\int_0^T\partial_{ij}H(t,X,(\widetilde A,B))\ud[X_i,X_j](t)+\sum_{i=0}^{m+\nu}\int_0^T\mathscr{D}_i H(t,X,(\widetilde A,B))\ud A_i(t).
 \end{split}
 \end{equation}

 Finally, plugging the results of \eqref{term 1}, \eqref{term 2}, \eqref{Horizontal H terms}, and \eqref{term 3} into \eqref{G Ito} yields that 
 \begin{align*}
  \int_0^T\eta(t,Y)\ud Y(t)&=H(T,X,(\widetilde{A},B))-H(0,X,(\widetilde{A},B))-\sum_{i=0}^{m+\nu+p}\int_0^T\mathscr{D}_i H(t,X,(\widetilde A,B))\ud A_i(t)\\
  &\qquad -\frac12\sum_{i,j=1}^d\int_0^T\partial_{ij}H(t,X,(\widetilde A,B))\ud[X_i,X_j](t),
 \end{align*}
 which by Theorem~\ref{cvfcont} and \eqref{H gradient} is equal to  $$\int_0^T\nabla_X H(t,X,(\widetilde A,B))\ud X(t)=\int_0^T\sum_{\ell=1}^\nu \eta_\ell(s,Y)\xi_{(\ell)}(s,X)\ud  X(s).$$
This concludes the proof.\end{proof}

\begin{proof}[Proof of Corollary~\ref{ass cor}] That the expression in 
\eqref{ass cor func int eq} is an admissible functional integrand follows immediately from Theorem~\ref{assoccont}  (b). Next, we define $B\in\CBV([0,T],\mathbb R^\nu)$ through
\begin{align*}
 B_\ell(t)=\sum_{i=0}^m\int_0^t\mathscr D_iF_\ell(s,X,A)\ud A_i(s)+\frac12\sum_{i,j=1}^d\int_0^t\partial_iF_\ell(s,X,A)\partial_jF_\ell(s,X,A)\ud[X_i,X_j](s),
\end{align*}
where $\ell=1,\dots,\nu$. When letting $\wt Y(t):=Y(t)-B(t)$, then \cite[Remark 8]{Schied13} yields that $\wt Y$ admits the continuous quadratic variation $[\wt Y]=[\blue{ Y}]$. Moreover, $\wt Y$ is of the form 
$$\wt Y_{\ell}(t)=\int_0^t\xi_{(\ell)}(s,X)\ud X(s),\qquad\ell=1,\dots,\nu,
$$
for $\xi_{(\ell)}(t,X):=\nabla_X F_\ell(t,X,A)$.
We next let $\wt A=(A,  B)$, $\wt F_\ell(t,X,\wt A):=F_\ell(t,X,A)-B_\ell$, and $\wt \eta(t,Z)=\eta(t,Z+B)$ for $Z\in D([0,T],\mathbb R^\nu)$. Then $\wt\eta$ is an admissible functional integrand. To see this, let $W\in\mathbb R^p$ be open and $G\in\mathbb C^{1,2}_c(\mathbb R^\nu,W)$ be such that $\eta(t,Z)=\nabla_ZG(t,Z,C)$ for some $C\in\CBV([0,T],W)$. We let $\wt G(t,Z,(C,B)):=G(t,Z+B,C)=G(t,\Phi(\cdot,Z,B),C)$ for $\Phi(t,Z,B)=Z(t)+B(t)$. Lemma
\ref{chain rule} now easily yields that $\wt G\in\blue{\mathbb C^{1,2}_c}(\mathbb R^\nu,W\times \mathbb R^\nu)$ and that $\nabla_Z\wt G(t,Z,(C,B))=\wt\eta(t,Z)$, and so $\wt\eta$ is indeed an admissible functional integrand.

Next,
 Lemma~\ref{Ito Stieltjes lemma}  yields that
\begin{align*}
\int_0^T\eta(s,Y)\ud Y(s)&=\int_0^T\wt \eta(s,\wt Y)\ud \wt Y(s)+\int_0^t\eta(s,Y)\ud B(s).
\end{align*}
Applying Theorem~\ref{assoccont} to the first integral on the right-hand side gives
\begin{align*}
\int_0^T\wt \eta(s,\wt Y)\ud \wt Y(s)&=\int_0^t\sum_{\ell=1}^\nu\wt \eta_\ell\big(s, \wt F_\ell(s,X,A)\big)\xi_{(\ell)}(s,X)\ud X(s)\\
&=\int_0^t\sum_{\ell=1}^\nu  \eta_\ell\big(s,   F_\ell(s,X,A)\big)\nabla_XF_{\ell}(s,X)\ud X(s).
\end{align*}
Using the definition of $B$ and applying once again the associativity of the Lebesgue--Stieltjes integral now yields the result.
\end{proof}

\noindent{\bf Acknowledgement:}  I.V.~gratefully acknowledges funding from Deutsche Forschungsgemeinschaft (DFG) through RTG 1953. A.S.~gratefully acknowledges partial  supported by DFG through RTG 1953 and  by NSERC through a Discovery Grant. Both authors thank an anonymous referee for pointing out several errors in a previous version of this manuscript.

 \parskip-0.5em\renewcommand{\baselinestretch}{0.9}
\end{document}